\documentclass[11pt]{article}
\usepackage{amsfonts}
\usepackage{amssymb}
\usepackage{amsmath}
\usepackage{amsthm}
\theoremstyle{plain}
\newtheorem{thm}{Theorem}[section]
\newtheorem{cor}[thm]{Corollary}
\newtheorem{lem}[thm]{Lemma}

\theoremstyle{definition}
\newtheorem{defn}{Definition}[section]

\theoremstyle{remark}

\numberwithin{equation}{section}
\title{ \bf Generalization of Kuratowski problem in vector spaces}
\author{ Allahkaram Shafie \\
 E-mail:  shafie.allahkaram@gmail.com}
\date{}

\begin{document}
\setlength{\evensidemargin}{-0.2in}
\maketitle
%\begin{center}\setlength{\columnsep}{1.5cm}

%{\bf\large Abstract}\\
%\end{center}
\noindent {\bf\large Abstract.}In this short paper,  Kuratowski  problem will be investigated in vector space. The highest number of distinct sets that can be generated from one convex set in linear space by repeatedly applying algebraic closure and complement in any order is 8. \vspace{2mm}

 \section{Introduction}

 In point-set topology Kuratowsky's closure-complement problem asks for the largest number of distinct sets obtained by repeatedly applying the set operations of closure and complement to a given starting subset of a topological space. The answer is 14. This result was first published by Kazimierz Kuratowski in 1992 \cite{KK}. The problem gained wide exposure three decades later as an exercise in John L.Kelley's classic textbook General topology \cite{KJ}.\\ In this article we will see what happens when topological space $X $ and topological closure replaced by vector space (without any topology) and algebraic closure.
In this paper $X$ is a real linear space and $A$ is a convex subset of $X.$
Let $f$ denote the algebraic closure operation and let $g$ denote the complement operation. Let $W$ be the set of strings (finite ordered lists) using only $f$ 
and $g$ For a fixed subset $A$ of a linear space $X$  and $w\in W$ let 
$wA$ denote the set obtained by applying the operations listed in $w$ from right to left. For example, if 
$w=fgg$ then $wA$  is the set obtained by first taking the complement of $A$  then taking the complement of that, and then taking the closure of that. For a fixed subset $A$ of a linear space $X.$ To simplify we consider $m(A)=\{wA:~~w\in W\}.$ It is not even clear that $m(A)$
 is finite, but we can simplify matters somewhat by reducing the number of strings we need to consider
but we can simplify matters somewhat by reducing the number of strings we need to consider. In the example above, we took the complement of $A$ then took the complement of that; but that's just $A$ In general, we see that $gg$ has no effect on the set. Therefore, if $w\in W$  contains a pair $gg $ we can remove it from the string without changing the set $wA.$ We have a similar simplification for $f$ namely $ff=f$ . With these two pieces of information at our disposal, we see that we only need to consider strings where there are no consecutive $f$'s or $g$'s; from this point on, $W$   will denote the set of such strings. While these restrictions are substantial, we are still left with infinitely many strings to consider:

\[f,g,fg,gf,fgf,gfg,gfgf,\]
Let us give a simple example of a set $A$ in topological space $X$ which from 
 the operation closure-complement the fourteen different sets will be obtained. 
\begin{defn}\label{d1}\cite{ABT} Let $S$ be a nonempty subset of vector space $X.$\\ 
(a) The set $${\rm cor}(S)=\{x\in X;\forall y\in X
~\exists\bar{\lambda}>0, x+\lambda y\in S,\forall\lambda\in[0,\bar{\lambda}]\}$$ is called
 the algebraic interior of $S.$\\
 (b) The set $S$ with $S={\rm cor}S$ is called algebraic open.\\
 (c) An element $\bar{x}\in X$ is called linearly accessible from $S$, if there
is an $x\in S,~x\neq\bar{x} $, with the property $$\lambda
x+(1-\lambda)\bar{x}\in S\quad\mbox{for all}~\lambda\in[0,1]$$The
union of $S$ is called the set of all linearly accessible elements
from $S$ is called the algebraic closure of $S$ and it is denoted by
$${\rm lin} (S):= S\cup\{x\in X| x~\mbox{is linearly accessible from
S}\}.$$In the case of $S={\rm lin}(S)$ the set $S$ is called algebraically
closed.\\ (d) The set $X^{'}$ is defined to be the set of all linear
maps from $X$ to $\mathbb{R}$ and it is called the algebraic dual
space of $X.$
\end{defn}
\begin{lem}\label{l1}\cite{ABT}
For a nonempty convex subset $S$ of a linear space we have:\\
(a) ${\rm cor}( {\rm cor}(S))={\rm cor}(S)$\\
(b) ${\rm cor}(S)\neq\emptyset\Rightarrow {\rm lin}({\rm cor}(S))={\rm lin}(S)$ and
${\rm cor}({\rm lin}(S))={\rm cor}(S)$
\end{lem}

\begin{thm}\label{t1}\cite{ABT}
Let $S$ and $T$ be nonempty convex subsets of a real linear space
$X$ with ${\rm cor}(S)\neq\emptyset.$ Then ${\rm cor}(S)\cap T=\emptyset$ if and
only if there are a linear functional $l\in
X^{'}\setminus\{0_{X^{'}}\}$ and a real number $\alpha$ with
$$l(s)\leq \alpha\leq l(t),~~\forall s\in S, ~t\in T$$ and $$l(s)<\alpha,~~\forall s\in {\rm cor}(S)$$
\end{thm}
\begin{cor}\label{c1}
Let $S$ be a nonempty and convex subset of a real linear space $X.$
Then $x\notin {\rm cor}(S)$ if and only if there are a linear functional
$l\in X^{'}\setminus\{0_{X^{'}}\}$ and a real number $\alpha$ with
$$l(s)\leq \alpha\leq l(x),~~\forall s\in S$$ and $$l(s)<\alpha,~~\forall s\in {\rm cor}(S)$$
\end{cor}

\section{Main results}
\begin{lem}\label{l2}
If $A\subset X$ be convex set then $X\setminus {\rm cor}(S)={\rm lin}(X\setminus
S)$
\end{lem}
\begin{proof}
The $\supseteq$ is trivial so we prove reciprocal  that is $X\setminus
{\rm cor}(S)\subset {\rm lin}(X\setminus S)$ and suppose that contrary $x\in X\setminus
{\rm cor}(S)$  and $x\notin{\rm lin}(X\setminus S).$ By corollary \eqref{c1} there exists $l\in
X^{'}\setminus\{0_{X^{'}}\}$ and a real number $\alpha$ with
\begin{eqnarray}\label{r1}l(s)\leq \alpha\leq l(x),~~\forall s\in S \end{eqnarray} and
$$l(s)<\alpha,~~\forall s\in {\rm cor}(S)$$  then $x\in S$ and for all $y\neq x, y\in
X\setminus S$ there exists $\lambda\in( 0,1]$ such that $\lambda
y+(1-\lambda)x\in S$ so by \eqref{r1} one has $l(x)=\alpha$ and
\begin{eqnarray*}l(\lambda y+(1-\lambda)x)=\lambda l(y)+(1-\lambda)l(x)=\lambda l(y)+(1-\lambda)\alpha\leq
\alpha\end{eqnarray*} therefore $l(y)\leq \alpha$ hence adding by \eqref{r1} derives that,  for all $x\in X,$ $l(x)\leq \alpha$ thus
$$l(x)\leq \frac{\alpha}{n},~~ \forall n\in\mathbb{N}$$ which
implies that $l(x)=0$
 but this is contradiction because of $l\neq 0.$ \end{proof}
 \begin{thm}
 Consider the collection all of convex subsets $A$ of the linear space $X.$
 The operator of algebraic closure $A\longrightarrow {\rm lin}(A)$ and
 complementation $A\longrightarrow X-A$ are functions from this
 collection of itself. then starting with a given set $A,$ one can
 form no more than $8$ distinct set by applying these two operations
 successively.
 \end{thm}
\begin{proof}
The notations algebraic closure,complement and algebraic interior will be denoted, respectively by $f,g,h.$ 
  According to Lemma \eqref{l1} $(b)$ we have $fhA=fA,hfA=hA$ which implies that $ffA=fA.$
  By using Lemma \eqref{l2} we have $hA=gfgA$ and since $fA$ is convex also one has $hfA=gfgfA,$ which implies that
  \[\begin{array}{l}
 A \to f \to gfA \to fgfA \to gfgfA \to fgfgfA = fhfA = fA,... \\ 
 A \to gA \to fgA \to gfgA \to fgfgA = fhA = fA,... \\ 
 \end{array}\]
  hence by given any subset $A$ of linear space $X$ there are at most $8$ distinct sets that can be produced
by taking algebraic closures and complements of A as
$$A,gA,fA,gfA,fgA,gfgA,fgfA,gfgfA.$$

\end{proof}

\end{document}